\documentclass{amsart}
\usepackage{amssymb}
\theoremstyle{plain}

\newtheorem{thm}{Theorem}[section]

\newtheorem{prop}{Proposition}[section]

\theoremstyle{definition}

\newtheorem{dfn}{Definition}[section]
\numberwithin{equation}{section}

\input amssym.def
\input amssym
\begin{document}

\author{Vladimir Bo\v zin,  Miodrag Mateljev\textrm{i\'{c}}}
\address{Faculty of mathematics, Univ. of Belgrade, Studentski Trg 16, Belgrade, Serbia}
\email{\rm bozin@matf.bg.ac.rs}
\email{\rm miodrag@matf.bg.ac.rs}
\title[Bounds
for  Jacobian]{Bounds for Jacobian of harmonic injective
mappings in n-dimensional space}

\maketitle

\thanks{Research partially supported by MNTRS, Serbia,  Grant No.
174 032}

\bigskip
\bigskip

\begin{abstract}
Using normal family arguments, we show that the degree of the first nonzero homogenous polynomial in the expansion of $n$ dimensional Euclidean harmonic $K$-quasiconformal mapping
around an
internal point is odd, and that
such a map
from the unit ball onto  a
bounded convex  domain,  with  $K< 3^{n-1}$,
is  co-Lipschitz. Also some generalizations  of this result are given, as well as a generalization of Heinz's lemma for harmonic quasiconformal maps in $\mathbb R^n$ and related results.

\end{abstract}
\section{Introduction}

In his seminal paper, Olli Martio \cite{OM1} observed that every
quasiconformal harmonic mapping of the unit planar disk onto
itself is co-Lipschitz. Later, the subject of quasiconformal
harmonic mappings was intensively studied by the participants of
the Belgrade Analysis Seminar, see for example \cite{kalaj.thesis, mama,
mm.fil12a,topic, rckm0,MP}. Harmonic quasiconformal maps have found applications in Teichm\" uller theory, among other things.
Recently, V. Markovi\' c  \cite{mar} proved  that a quasiconformal map of
the sphere $\mathbb{S}^2$ admits a harmonic quasi-isometric
extension to the hyperbolic space $\mathbb{H}^3$, thus confirming
the well known Schoen Conjecture in dimension $3$.
Related questions of bi-Lipschitzity and bounds of Jacobian have been studied in a sequence of papers by Kalaj and Mateljevi\'c; see also a recent paper of
Iwaniec-Onninen  \cite{IwOn}. The corresponding results for
harmonic maps between surfaces were  obtained previously by  Jost
and Jost-Karcher  \cite{jost2,jost}.  In the planar case, the
complex  harmonic function $h$ on a simply connected planar domain
can be written   in the form    $h=f+\overline{g}$, where $f$ and
$g$ are   holomorphic, so that  $|f'| $  satisfies the minimum
principle  and Lewy's theorem.  There is no appropriate analogy in
higher dimensions; if $h$ is a harmonic mapping from a domain in
$\mathbb{R}^{n}$   to $\mathbb{R}^n$ then   $\|\partial_j h\| $ is
subharmonic,   but it does not satisfy the  minimum principle in
general. In fact,  Lewy's theorem is false in dimensions higer
than two (see \cite{wood}, also \cite{dur}  pp. 25-27 for Wood's
counterexample). In a very special case, gradients of harmonic
functions in $\mathbb{R}^3$, for which Lewy's theorem is true,
that also turn out to map unit ball onto a convex domain, are
known to be co-Lipschitz (see Astala-Manojlovi\' c \cite{ast.ma}, Mateljevi\'c \cite{rckm0}). However, it seems
that in general, one needs a different approach in higher
dimensions.

For example, in \cite{KaMpacific} the following general theorem was proved:
\begin{thm}
A $K$-quasiconformal harmonic mapping $f$ of the unit $n$ dimensional ball
 ($n>2$) onto itself is Euclidean
bi-Lipschitz, provided that $f(0) = 0$ and that $K<2^{n-1}$,
where $n$ is the dimension of the space.
\end{thm}
It is an extension of a similar result for hyperbolic harmonic
mappings with respect to the hyperbolic metric (see Tam and Wan,
\cite{tw}, 1998). The proof makes use of M\"obius transformations
in the space, and of a recent  Kalaj's result \cite{Ka.janal13},
which states that harmonic quasiconformal self-mappings of the
unit ball are Lipschitz continuous.

Among other things, in this  paper we prove that the above result holds if  $K<3^{n-1}$, and when codomain is only assumed to be convex. A suitable application of normal family
argument allows us to take a conceptually simpler approach   then
in \cite{KaMpacific}.

The proof is based on  Theorem  \ref{t.locJac}, showing that the degree of the first nonzero homogenous polynomial in the expansion of $n$ dimensional Euclidean harmonic quasiconformal mapping around an internal point is odd. We combine  this  with a distortion property of quasiconformal maps to
prove that  for $n$ dimensional Euclidean harmonic quasiconformal mappings with
$K_O(f) < 3^{n-1}$, Jacobian is never zero.

Our approach gives motivation to define Jacobian non-zero closed families (see  Definition \ref{dfn3.1}) for which a generalization of Heinz's lemma is shown; we also prove bounds for Jacobian from above for arbitrary harmonic quasiconformal maps. Generalization of Heinz's lemma allows us to prove Theorem \ref{thm.space2}, namely that  harmonic quasiconformal maps from unit ball onto a convex domain, that are from Jacobian non-zero closed families, are co-Lipschitz. Essentially, we show that if a map is not co-Lipschitz, then one can get a map of the same type, for which Jacobian vanishes at some internal point. Several applications are also given.

The content of the paper is as follows.   In section \ref{sec2} we
collect some  known definitions and results which we use  in the
paper. Proofs that Jacobians of quasiconformal harmonic maps cannot vanish when $K_O<3^{n-1}$ or in the case of gradients of harmonic functions, and that Taylor expansions of quasiconformal harmonic maps have odd lowest degree are
subject of section \ref{sec4}.
In section \ref{sec3}  we prove the generalization of Heniz's lemma and the co-Lipschitz properties described above, and related results.

\section{Background and Auxiliary results}\label{sec2}

Throughout this paper, we will consider maps from domains, i.e. open and connected regions, of $\mathbb{R}^n$, usually denoted by $\Omega,\Omega'$, to $\mathbb{R}^n$. We will use notation $\mathbb{B}^n$ for the unit ball in $\mathbb{R}^n$, and for $x\in \mathbb{R}^n$ its norm will be denoted by $\|x\|$. For $x\in \mathbb{R}^n$, $a>0$ and $A \subseteq \mathbb{R}^n$, by $d(x,A)$ we will denote the Euclidean  distance of point $x$ from the set $A$, by $aA$ the set $\{a y\,| \,y\in A\}$ and by $x+A$ the set $\{x+y\,|\, y\in A\}$.

By $J_f(x)$ we will denote the Jacobian of $f$ at point $x$, $\partial_j f$ will stand for $\frac{\partial f}{\partial x_j}$, and $\partial_{ij}^2 f$ for $\frac{\partial^2 f}{\partial x_i \partial x_j}$, where $x=(x_1, x_2, \ldots, x_n)$ is the vector argument of $f$.

We will consider Euclidean harmonic maps, also called harmonic maps in this paper, i.e. those with zero Laplacian of each coordinate function. Also, we will deal with quasiconformal maps.

For a domain $\Omega$  in $\mathbb{R}^{n}$, a map $f:\Omega \mapsto \mathbb{R}^{n}$ is quasiconformal if it is a homeomorphism of $\Omega$ to $f(\Omega)$, and if $f$ belongs to Sobolev space $W_{1,\,loc}^{n}(\Omega )$ and there exists $K$, $1\leq K<\infty $, such that
\begin{equation}
\|Df(x)\|^{n}\leq K\,J_{f}(x)\,\,\,\textrm{a.e. on }\Omega, \label{outer.qr0}
\end{equation}
where $\|Df(x)\|$ denotes the operator norm of the Jacobian matrix of $f$ at $x$. The smallest $K$ in (\ref{outer.qr0}) is called the outer dilatation $K_{O}(f)$. The inner dilatation $K_{I}(f)=K_{O}(f^{-1})$, and map $f$ is $K$-quasiconformal if $\max (K_O,K_I) \leq K$.

We will need the following proposition concerning a distortion
property of quasiconformal mappings (see \cite{vu}):
\begin{prop}\label{p-dist} If  $g:\mathbb{B}^n \mapsto \mathbb{B}^n$ is quasiconformal,
$g(0)=0$  and $1/\alpha =K_I(g^{-1})^{1/(n-1)}$, then for some $m>0$, $\|g(x)\|\geq m \|x\|^{1/\alpha}.$
\end{prop}

The next theorem concerns harmonic maps onto a convex domain. For the planar version  of Theorem \ref{thm.space1} cf.
\cite{revroum01,napoc1}, also \cite{topic}, pp.~152-153.  The space  version was  communicated  on
International Conference on Complex Analysis and   Related Topics
(Xth Romanian-Finnish Seminar, August 14-19, 2005, Cluj-Napoca,
Romania), by Mateljevi\'c, cf. also \cite{rckm0}. For convenience of the reader, we  repeat the proof.

\begin{thm}\label{thm.space1}
Suppose  that   $h$ is  an Euclidean harmonic mapping
from the unit ball $\mathbb{B}^n$ onto  a
bounded convex  domain  $D=h(\mathbb{B}^n)$, which
contains the  ball  $h(0)+R_0 \mathbb{B}^n$.  Then for any $x\in \mathbb{B}^n$
$$d(h(x),\partial D) \geq (1-\|x\|) R_0/2^{n-1}.$$
\end{thm}

\begin{proof}
To every  $a\in \partial D$  we associate
a nonnegative harmonic function $u=u_a$. Since  $D$ is convex, for
$a\in
\partial D$, there is a  supporting hyper-plane $\Lambda_a$, defined as set of all $y$ for which $(y-a,n_a)=0$, where
$n_a$  is  a  unit vector such that
$(y-a,n_a)\geq 0$ for every $y\in \overline{D}$.

Define $u(x)=(h(x)-a,n_a)$. Since $n_a$ is a unit vector, $u(x)\leq \|h(x) -a\|$. Then
$u(0)=(h(0)-a,n_a)= d(h(0),\Lambda_a)$. From geometric
interpretation it is clear that $d(h(0),\Lambda_a) \geq R_0$.

By  Harnack's inequality (cf. \cite{gtrudi}, p. 29), ${c}_n (1-\|x\|)  u(0)  \leq u(x)$,
where ${c}_n=2^{1-n}$.  In particular,
${c}_n (1-\|x\|) R_0 \leq u(x)\leq \|h(x) -a\|$   for every  $a \in \partial D$. Hence, for a
fixed $x\in \mathbb{B}^n$,  $d(h(x),\partial D)=\inf_{a \in \partial D}\|h(x) -a\| \geq
{c}_n (1-\|x\|) R_0$  and therefore  we obtain  the required inequality.
\end{proof}

To apply normal family arguments, we need  the following
results; see Vaisala \cite{vaisala}.
\begin{thm}
Suppose that $\Omega$ is a domain in $\overline{\mathbb{R}^n}$, that $K\geq 1$ and that $r>0$. If $\mathcal{F}$ is a family of $K$-quasiconformal mappings of $\Omega$ (not necessarily onto a fixed domain), such that each $f\in \mathcal{F}$ omits two points $a_f, b_f$ whith spherical distance in $\overline{\mathbb{R}^n}$ at least $r$, then $\mathcal{F}$ is a normal family.
\end{thm}

\begin{thm}\label{nrm}
Let $(f_{j})$, $f_j: \Omega \mapsto \overline{\mathbb{R}^n}$,
be a sequence of $K$-quasiconformal maps, which converges pointwise
to a mapping $f: \Omega \mapsto \overline{\mathbb{R}^n}$.
Then there are three
possibilities:\\
A.   $f$ is a homeomorphism  and the convergence is uniform on compact sets.\\
B.   $f$ assumes exactly two values,  one of which at exactly one
point; covergence is not uniform on compact sets in that case.\\
C.  $f$ is constant.
\end{thm}

Note that the case B does not happen when we use normal families.

\section{Interior zeros of the Jacobian}\label{sec4}
In this section we prove that a quasiconformal harmonic map cannot have lowest degree polynomials in the Taylor expansion of even degree.
Because harmonic functions are real analytic, in the neighborhood of a zero of the Jacobian there is a power expansion in terms of coordinates. The following proposition follows directly
from the quasiconformality condition:

\begin{prop}\label{p.same}
Suppose that $h$ is a real analytic quasiconformal mapping
from a domain $\Omega \subset\mathbb{R}^n$ to
$\mathbb{R}^n$, such that Jacobian is zero at some point $x_0$. Then all the
degrees of first non-zero homogenous polynomials in the Taylor expansion of the coordinate functions of $h$
around $x_0$ are   the same.
\end{prop}

Now the following theorem holds:
\begin{thm}\label{t.locJac}
Suppose that $h$ is an Euclidean harmonic quasiconformal mapping
from a  domain $\Omega \subset\mathbb{R}^n$  to
$\mathbb{R}^n$, such that Jacobian is zero at $x_0\in \Omega$. Then the
degree of first non-zero homogenous polynomials in the Taylor expansion of  $h$
around $x_0$ is odd, and the corresponding homogenous polynomial map, obtained by taking the lowest degree homogenous polynomials in the Taylor expansion of the coordinates, is harmonic and quasiconformal.
\end{thm}

\begin{proof}
Without loss of generality, by restricting to a ball neighbourhood
and a suitable change of variable, we may assume that $x_0=0$ and
that $\Omega=\mathbb{B}^n$. Suppose the contrary, that the lowest
degree of the first non-zero homogenous polynomials in the Taylor
expansion of the coordinates of $h$ is even, say equal to $2m$.
Then consider the sequence of harmonic quasiconformal maps, $h_j:
\mathbb{B}^n \mapsto \mathbb{R}^n$, $h_j(x)=j^{2m} h(x/j)$. Note
that the first non-zero homogenous polynomials in the expansion of
all the maps $h_j$ are the same as for $h$. Because for $j>1$
derivatives of $h$ are bounded uniformly in $j$ and on
$\mathbb{B}^n$, from Taylor expansion it follows that all the maps
$h_j$, for $j>1$, are uniformly bounded on the unit ball.
Therefore, $\{h_j \,|\, j\in \mathbb{N}\}$ is a normal family, and
a subsequence of our sequence converges to a harmonic mapping $f$
uniformly on compact sets. By elliptic regularity theory (cf. H\"
older and Schauder apriori estimates, \cite{gtrudi}, pp. 60, 90),
all the derivatives of the subsequence will converge to the
corresponding derivatives of $f$. It follows that coordinates of
$f$ are equal to homogenous polynomials of degree $2m$, since
higher degree homogenous polynomials in the expansions of $h_j$
tend to zero. In particular, the limit function $f$ is not
constant, and by Theorem   \ref{nrm}, $f$ is quasiconformal, and hence injective. But the
limit map $f$ satisfies $f(-x)=f(x)$, which is a contradiction. Similar procedure in the case of odd lowest degree gives the claimed
homogenous polynomial harmonic quasiconformal map.
\end{proof}

Next, we will combine this theorem with a distortion property of quasiconformal maps.
\begin{prop}\label{p-tom2}
 Suppose that  $h:  \Omega\mapsto \mathbb{R}^n$  is  a harmonic  quasiconformal map. If $\partial_j h(x_0)=0$ and $\partial^2_{ij}h(x_0)=0$ for some $x_0\in \Omega$,
then $K_O(h)  \geq 3^{n-1}$ .
\end{prop}

\begin{proof} Without loss  of  generality, by restricting to ball neighbourhood of $x_0$ whose closure is in the domain, and a change of variable which does not change quasiconformal distortion,  we can suppose that $x_0=0$, $h(x_0)=0$, and that, by the Taylor formula,  there is  $M>0$   such  that $\|h(x)\|
\leq M \|x\|^3$ on $\mathbb{B}^n$. Let $g=(h|_{\mathbb{B}^n})/M$

If  $1/\alpha =K_I(g^{-1})^{1/(n-1)}$, then by Proposition \ref{p-dist}, $m \|x\|^{1/\alpha}\leq \|g(x)\|\leq
\|x\|^3$.   Hence   $K_O^{1/(n-1)} \geq 3$, and therefore  $K_O
\geq 3^{n-1}$, where $K_O(g)=K_I(g^{-1})$.
\end{proof}

\begin{thm}\label{p-tom3}

 Suppose that  $h:  \Omega\mapsto \mathbb{R}^n$  is  a harmonic  quasiconformal map. If  $K_O(h) <
3^{n-1}$, then  its Jacobian has no zeros.
\end{thm}
\begin{proof}
Contrary suppose that $J_h(x_0)=0$ for some $x_0\in \Omega$. Since $h$ is quasiconformal,
by Theorem  \ref{t.locJac} we find  $\partial^2_{ij}h(x_0)=0$.
Now, by Proposition \ref{p-tom2}, $K_O(h)  \geq 3^{n-1}$  and this
yields a contradiction.
\end{proof}

\section{Bounds
for the Jacobian}\label{sec3}

Heinz's lemma-type results can be obtained for quasiconformal harmonic maps, using normal families. To state our results clearly and in their generality, it is useful to give some definitions.

\begin{dfn}\label{dfn3.2}
We say that a family $\mathcal F$ of maps from domains in $\mathbb R^n$ to $\mathbb R^n$ is RHTC-closed if the following holds:
\begin{itemize}
\item (Restrictions) If $f: \Omega\mapsto \mathbb{R}^n$ is in $\mathcal{F}$, $\Omega'\subset \Omega$ is open, connected and nonempty, then $f|_{\Omega'}\in \mathcal{F}$.
\item (Homothety)  If $f: \Omega\mapsto \mathbb{R}^n$ is in $\mathcal{F}$, $a\in \mathbb R$, $a>0$ then $g: \Omega\mapsto \mathbb{R}^n$ and $h: a \Omega\mapsto \mathbb{R}^n$ are in $\mathcal{F}$, where $g(x)=a f(x)$ and $h(x)=f(x/a)$.
\item (Translations)  If $f: \Omega\mapsto \mathbb{R}^n$ is in $\mathcal{F}$, $t\in \mathbb R^n$,  then $g: \Omega\mapsto \mathbb{R}^n$ and $h: t+ \Omega\mapsto \mathbb{R}^n$ are in $\mathcal{F}$, where $g(x)=t+ f(x)$ and $h(x)=f(x-t)$.
\item (Completeness)  If $f_j: \Omega\mapsto \mathbb{R}^n$, $j\in \mathbb{N}$ are in $\mathcal{F}$, $(f_j)$ converges uniformly on compact sets to $g: \Omega\mapsto \mathbb{R}^n$, where $g$ is non-constant, then $g\in \mathcal{F}$.
\end{itemize}
\end{dfn}
For instance, families of harmonic maps and of gradients of harmonic functions are RTHC-closed. Also, due to Theorem $\ref{nrm}$, for any given $K\geq 1$, a subfamily of $K$-quasiconformal members of a RTHC-closed family is also RTHC-closed.

\begin{dfn}\label{dfn3.1}
We say that a family $\mathcal F$ of harmonic maps from domains in $\mathbb R^n$ to $\mathbb R^n$  is non-zero Jacobian closed, if it  is RHTC-closed and Jacobians of all maps in the family have no zeros.
\end{dfn}

Note that uniform convergence on compact sets in the case of harmonic maps implies convergence of higher order derivatives, via H\" older and Schauder apriori estimates (see \cite{gtrudi}, pp. 60, 90). This is related to elliptic regularity and holds for more general elliptic operators, and not just Laplacian, so that this method applies in that more general setting too.

\begin{thm}\label{ghall}
For every non-zero Jacobian closed family of $K$-quasiconformal harmonic maps, there is a constant $c>0$, such that if $f:\mathbb{B}^n \mapsto  \mathbb{R}^n$ is from the family,  $d(0,\partial f(\mathbb{B}^n)) \geq 1$  and $f(0)=0$, then
$$J_f(0)\geq c.$$
\end{thm}

\begin{proof}
Suppose the contrary, i.e. that the family contains a sequence
$(f_j)$ of maps from the unit ball satisfying
$\mathbb{B}^n\subseteq f_j(\mathbb{B}^n)$ and $f_j(0)=0$, such
that $J_{f_j}(0) \rightarrow 0$ as $j\rightarrow \infty$.
Multiplying functions by constants less than $1$ if necessary, we
may assume, without loss of generality, that the boundary of the
image $f_j(\mathbb{B}^n)$ always contains a point on the unit
spere, and thus use a normal family argument (since infinity and
point on the unit sphere are on a fixed spherical distance) to
pass to a convergent subsequence.  Now note that because of the
Gehring distortion property
 (see \cite{vaisala} p. 63, \cite{aaa}), $f_j(\frac{1}{2}\mathbb{B}^n)$ will contain a ball around zero of fixed radius, so the limit cannot degenerate to a constant function. But then the limit, say $g$, is in the family. However, by apriori estimates of elliptic regularity theory, derivatives of $f_j$ will also converge to derivatives of $g$, and hence $J_g(0)=0$, contradicting the non-zero Jacobian assumption.
\end{proof}

Note that the same normal family argument gives upper bound for Jacobian, but for general quasiconformal harmonic maps. Namely, the following theorem holds:

\begin{thm}\label{upperl}
There is a constant $c>0$, depending only on $K$, such that if $f:\mathbb{B}^n \mapsto  \mathbb{R}^n$ is K-quasiconformal harmonic, $d(0,\partial f(\mathbb{B}^n))\leq 1$ and $f(0)=0$, then
$$J_f(0)\leq c.$$
\end{thm}

\begin{proof}
Proof is essentially the same as for Theorem  \ref{ghall}: we take a sequence of $K$-quasiconformal harmonic maps $(f_j)$ from the unit ball, such that $J_{f_j}(0) \rightarrow \infty$ as $j\rightarrow \infty$,  $f_j(0)=0$ and $d(0,\partial f_j(\mathbb{B}^n))=1$, multiplying by constants now greater than one if necessary. This will provide a subsequence with a limit mapping whose Jacobian at zero is finite, a contradiction.
\end{proof}

Applying Theorem \ref{thm.space1}, we get the following result regarding co-Lipschitz condition for maps from ball to a convex domain:

\begin{thm}\label{thm.space2}
Suppose that $h$  is  a harmonic quasiconformal mapping
from the unit ball $\mathbb{B}^n$ onto  a
bounded convex  domain  $D=h(\mathbb{B}^n)$, and that $h$  belongs to a non-zero Jacobian closed family of  harmonic maps. Then $h$ is co-Lipschitz on  $\mathbb{B}^n$.
\end{thm}
\begin{proof}
Let $x_0$ be a point in $\mathbb{B}^n$. Define $f:\mathbb{B}^n\mapsto \mathbb{R}^n$ by
$$f(x)=\frac{h(x_0+(1-\|x_0\|)x)-h(x_0)}{d(h(x_0),\partial D)}.$$
Applying Theorem \ref{ghall} to $f$, and using the fact that norm of the derivative of $K$-quasiconformal map is bounded from below by a constant times $n$-th root of the Jacobian, and using the uniform estimate $(1-\|x_0\|)/(d(h(x_0),\partial D))\leq 2^{n-1}/d(h(0),\partial D)$ from Theorem \ref{thm.space1}, we get a uniform bound from below for norm of the derivative of $h$, and hence conclude that map is co-Lipschitz.
\end{proof}

A special case of interest we get by applying Theorem \ref{thm.space2}, combining it with Theorem  \ref{p-tom3}.

\begin{thm}
Suppose $h$  is  a harmonic $K$-quasiconformal mapping
from the unit ball $\mathbb{B}^n$ onto  a
bounded convex  domain  $D=h(\mathbb{B}^n)$,  with  $K< 3^{n-1}$.
Then       $h$  is  co-Lipschitz on  $\mathbb{B}^n$.
\end{thm}

{}

\begin{thebibliography}{}
\smallskip


\bibitem{ast.ma} {\sc K. Astala, V. Manojlovi\' c}, {\it Pavlovic's theorem in space},
arXiv:1410.7575 [math.CV]

\bibitem{dur} {\sc  P. Duren},    {\it Harmonic   mappings in the plane},
Cambridge Univ.   Press,  2004.

\bibitem{gtrudi} {\sc D. Gilbarg, N. S. Trudinger}, {\it Elliptic Partial Differential Equations of Second Order}
Springer-Verlag Berlin Heidelberg New York Tokyo,  2001.

\bibitem{IwOn} {\sc T. Iwaniec and  J. Onninen}, {\it  Rad\'{o}-Kneser-Choquet
theorem}, Bull. London Math. Soc. 46 (2014) 1283-1291

\bibitem{jost2}
{\sc J. Jost}, {\it Harmonic Maps Between surfaces},  Springer-Verlag
Berlin Heidelberg New York Tokyo,  1984.

\bibitem{jost}
{\sc J. Jost}, {\it Two-dimensional Geometric Variational Problems},
John Wiley \& Sons, 1991.


\bibitem{KaMpacific} {\sc D. Kalaj, M. S. Mateljevi\' c},  {\it Harmonic quasiconformal
self-mappings and M\" obius transformations of the unit ball.}
Pacific journal of mathematics. ISSN 0030-8730. 247 : 2 (2010),
389-406.

\bibitem{kalaj.thesis} {\sc D.  Kalaj,} {\it Harmonic and
quasiconformal functions between convex domains}, Doctoral Thesis,
University of Belgrade,  submited 2000,  (2002).

\bibitem{Ka.janal13}    {\sc D. Kalaj},   {\it  A priori estimate of gradient of a solution to certain differential inequality and quasiregular mappings}, Journal
d'Analyse Math\'{e}matique, (2013), Volume 119, Issue 1, pp.
63-88.

\bibitem{mar} {\sc V. Markovi\' c} {\it   Harmonic maps between 3-dimensional hyperbolic spaces}.
 Inventiones Mathematicae, DOI
10.1007/s00222-014-0536-x, http://arxiv.org/pdf/1308.1710.pdf.

\bibitem{mama} {\sc V. Markovi\' c, M. Mateljevi\' c}, {\it A new version of the main inequality and the uniqueness of harmonic maps}, Journal d'Analyse Math\'{e}matique,
(1999), Volume 79, Issue 1, pp. 315-334.

\bibitem{OM1} {\sc O. Martio}, {\it On harmonic quasiconformal
mappings}, Ann. Acad. Sci. Fenn., A 1, 425, 3-10 (1968)

\bibitem{revroum01}
{\sc M. Mateljevi\' c,}   {\it Estimates for the modulus of the
derivatives of harmonic univalent mappings}, Proceedings of
International Conference on Complex Analysis and Related Topics
(IX$^{th}$ Romanian-Finnish Seminar, 2001),   Rev Roum  Math Pures
Appliq (Romanian Journal of Pure and Applied mathematics)  47
(2002) 5-6,   709 -711.

\bibitem{napoc1} {\sc M.  Mateljevi\'c,}  {\it Distortion  of harmonic functions and
harmonic quasiconformal  quasi-isometry},   Revue Roum. Math.
Pures Appl.  Vol. {\bf 51} (2006) 5--6, 711--722.

\bibitem{mm.fil12a} {\sc M. Mateljevi\' c,} {\it Quasiconformality of harmonic mappings between Jordan
domains 2},  Filomat  26:3  (2012) 479-509.


\bibitem{topic} {\sc M. Mateljevi\'{c}}, {\it Topics in Conformal,
Quasiconformal and Harmonic maps}, Zavod za ud\v{z}benike, Beograd
2012.

\bibitem{aaa} {\sc M. Mateljevi\' c,}  {\it Distortion of quasiregular
mappings and equivalent norms on Lipschitz-type spaces}, Abstract
and Applied Analysis Volume 2014 (2014), Article ID 895074, 20
\bibitem{rckm0}{\sc  M. Mateljevi\' c},    {\it The
lower bound for the modulus of the derivatives and Jacobian of
harmonic  injective  mappings}   http://arxiv.org/pdf/
1501.03197v2 [math.CV] 7 Feb 2015

\bibitem{MP} {\sc M. Pavlovi\' c},
{\it Boundary correspondence under harmonic quasiconformal
homeomorfisms of the unit disc}, Ann. Acad. Sci. Fenn., Vol 27,
(2002) 365-372.

\bibitem{tw} {\sc L. Tam,   T. Wan,}   {\it On  quasiconformal   harmonic    maps},
Pacific J  Math,   {\bf V}   182  (1998),   359-383

\bibitem{vaisala} {\sc J. V\"{a}is\"{a}l\"{a}}, {\it Lectures on
n-Dimensional Quasiconformal Mappings}, Springer-Verlag, 1971.
\bibitem{vu} {\sc M. Vuorinen}, {\it Conformal geometry and
quasiregular mappings}, Lecture Notes in Math. \textbf{V 1319},
Springer-Verlag, 1988.

\bibitem{wood} {\sc J. C. Wood}, {\it Lewy's theorem fails in higher dimensions}, Math. Scand. 69 (1991), 166.

\end{thebibliography}
\end{document}